\documentclass[a4paper,11pt]{amsart}
\usepackage{amsmath,amssymb,colordvi}
\usepackage{graphicx}

\newtheorem{theorem}{Theorem}
\newtheorem{corollary}[theorem]{Corollary}
\newtheorem{lemma}[theorem]{Lemma}
\newtheorem{example}[theorem]{\it Example}
\newtheorem{proposition}[theorem]{Proposition}

\newtheorem{remark}[theorem]{\it Remark}

\font\tenBb=msbm10 \font\sevenBb=msbm7 \font\fiveBb=msbm5
\newfam\Bbfam \textfont\Bbfam=\tenBb \scriptfont\Bbfam=\sevenBb
\scriptscriptfont\Bbfam=\fiveBb

\def\Bb{\fam\Bbfam\tenBb}
\def\R{{\Bb R}}
\def\C{{\Bb C}}
\def\N{{\Bb N}}

\def\D{{\mathcal D}_+}

\def\D{{\mathbb D}}

\renewcommand\Re{{\rm Re}}

\usepackage{ulem}
\usepackage{cancel}
\newcommand\Lip{{\rm Lip}}
\newcommand\lip{{\rm lip}}

\usepackage{tikz}
\usepackage{xparse}
\NewDocumentCommand\UpArrow{O{2.0ex} O{black}}{%
    \mathrel{\tikz[baseline] \draw [->, line width=0.5pt, #2] (0,0) -- ++(0,#1);}
}

\NewDocumentCommand\DownArrow{O{2.0ex} O{black}}{%
    \mathrel{\tikz[baseline] \draw [<-, line width=0.5pt, #2] (0,0) -- ++(0,#1);}
}

\begin{document}
\title{Boundary values of holomorphic semigroups and fractional integration}

 \author{Omar El-Mennaoui}
  \address{Universit\'e Ibn Zohr, Facult\'e des Sciences, D\'epartement de Math\'ematiques, Agadir, Maroc}
  \email{elmennaouiomar@yahoo.fr}

\author{Valentin  Keyantuo}
\address{University of Puerto Rico,  R\'io Piedras Campus, Department of Mathematics,
Faculty of Natural Sciences,  17 AVE Universidad STE 1701,  San
Juan PR 00925-2537 U.S.A.}
\email{valentin.keyantuo1@upr.edu}

\author{Ahmed Sani}
  \address{Universit\'e Ibn Zohr, Facult\'e des Sciences, D\'epartement de Math\'ematiques, Agadir, Maroc}
 \email{ahmedsani82@gmail.com }

\thanks{The work of the second   author  is partially supported by  Air Force Office of Scientific Research under the Award No:  FA9550-18-1-0242.}


\subjclass[2010]{Primary   47D06; Secondary 47D60  }

\keywords{Holomorphic semigroup, fractional powers, fractional integration, Hadamard type fractional integrals, little
H\"older spaces, Riemann-Liouville semigroup.}

\begin{abstract}
The concept of  boundary values of holomorphic semigroups in a general Banach space is studied.  As an application, we consider the  Riemann-Liouville  semigroup
 of  integration operator  in the little  H\"older spaces
$\rm{lip}_0^\alpha[0,\, 1] , \,  0<\alpha<1$ and prove that it
admits a strongly  continuous boundary group, which is the group of
fractional integration of purely imaginary order. The corresponding
result for the $L^p$-spaces ($1<p<\infty$) has been known for some
time, the case $p=2$ dating back to the monograph
 by Hille and Phillips. In the context of $L^p$ spaces, we establish the existence of the boundary group of the Hadamard fractional integration
  operators using  semigroup methods. In the general framework, using a suitable spectral decomposition,we give a partial treatment of the inverse problem, namely: Which $C_0$-groups are  boundary values of some holomorphic semigroup of angle $\pi/2$?

\end{abstract}

\maketitle

\section{Introduction}
\setcounter{theorem}{0}
 \setcounter{equation}{0}

The concept of boundary values of holomorphic semigroups that we
use in the present work originates in the treatise \cite{HP1957}. Specifically, in \cite[Chapter 17]{HP1957}, the authors give a necessary and sufficient condition
for a holomorphic semigroup of angle $\pi/2$ to admit a boundary value group. The converse question of which groups are boundary values of holomorphic semigroups  is answered in \cite[Theorem 17.10.1]{HP1957}.
Using this concept, Arendt, El-Mennaoui and Hieber \cite{AEH1997} gave an
elementary proof to the classical result of L. H\"ormander \cite[Theorem 3.9.4 and Chapter 8]{AEH1997} to the effect
that the Schr\"odinger operators $\pm i\Delta$ generate  $C_0-$semigroups on
$L^p(\mathbb{R}^n)$ if and only if $p=2.$ Later the same approach
was  also used in \cite{EK1996} to prove a similar result in $L^p(\Pi ^n)$
with Dirichlet, Neumann or periodic boundary conditions (where
$\Pi ^n=\mathbb{R}^n/\mathbb{Z}^n$ denotes the $n-$dimensional
torus).

The following proposition \cite[Theorem 17.9.1 and Theorem 17.9.2]{HP1957} answers the question of which holomorphic semigroups  in the right half-plane admit boundary value groups.

\begin{proposition} Let $A$ be the generator of a holomorphic $C_0-$semigroup $T$ of   angle $\pi/2$
 on a Banach space $E $.
Then $iA$ (resp. $-iA$) generates a $C_0-$semigroup (which is the
boundary value of $T$) if and only if $T$ is bounded on $D_+:=\{
z\in \mathbb{C};{\rm Re}(z)>0,\ {\rm Im}(z)>0,\ \vert z\vert \leq
1\}$ (resp. on $D_-:=\{ z\in \mathbb{C};{\rm Re}(z)>0,\ {\rm
Im}(z)<0,\ \vert z\vert \leq 1\}$).
\end{proposition}

This result was extended by El-Mennaoui to cover holomorphic semigroups which  admit as boundary values exponentially bounded integrated semigroups.
The results are presented in \cite[Section 3.14]{ABHN2011} and applied to the Schr\"odinger equation in $L^p$ spaces.

Our focus in the present paper is  on the Riemann-Liouville semigroup which
describes the fractional integration. It is also the basis for the definition of the most commonly used concepts of
fractional derivatives, namely the Riemann-Liouville and the Caputo fractional derivatives. The explicit representation
of this semigroup $(J(z))_{{\rm Re}(z)>0}$ is given by 
\begin{equation}\label{RL_representation}
 J(z)f(t)= \frac{1}{\Gamma(z)}\int_0^t(t-s)^{z-1}f(s) {\rm d}s\qquad (f\in L^1[0,1], t\in (0,1)).
\end{equation}
We will  show that the semigroup $J$ is bounded on
$D:=D_+\cup D_-$ with respect to the norm of the space ${\rm Lip}_\alpha[0,1]$ of all $\alpha-$Lipschitz continuous functions on $[0,1].$
In view of Proposition 1.1
we conclude the existence of a boundary value group on ${\rm lip}_\alpha[0,1]$ which is the  largest subspace of ${\rm Lip}_\alpha[0,1]$  on which $J$ is strongly continuous.\\

The Riemann-Liouville semigroup was studied extensively in \cite[Section 23.16]{HP1957} where it is proved that it is strongly continuous and holomorphic
in $L^p(0,\, 1),\, 1\le p<\infty.$ In particular, a description of its infinitesimal generator is given \cite[Section 23.16.1]{HP1957}. The importance of this semigroup
in spectral and ergodic theory is stressed and it is proved that when $p=2,$ a boundary group exists. The proof in that monograph relies on a lemma by Kober \cite[Lemma 23.16.2]{HP1957}. An interesting result is the explicit computation of the norm of boundary group  $\Vert J(it) \Vert_{L^2}=e^{\frac{1}{2}\vert t\vert},\, t\in \mathbb{R}. $ This means in particular that the boundary group is not uniformly bounded.\\
Let $-B$ be the generator of the translation semigroup on $L^p[0,1].$ In section 3 below we recall the well known fact that $J(z)$ may be identified with $B^{-z} \ ({\rm Re}(z)>0).$
It was proved in \cite{EK1996} (see also \cite[Section 3.14]{ABHN2011}), using
 the transference principle due to R. Coifman and G. Weiss \cite{C-W76}
 and/or \cite{C-W77}
that $J$ is  bounded on $D_\pm$  as a holomorphic semigroup  of
angle $\pi/2$ acting on $L^p[0,1]$ if and only if  $1<p<+\infty.$
Here the notation $D_{\pm}$ will stand henceforth for $D_+$ or $D_-$ and will mean either
the first object or the second one. A similar notation prevails for all mathematic objects along this paper like $E_{\pm}$ and $\pm i\Delta$ in section 2.
It should be noted that a similar result had appeared earlier  (see the two papers \cite{Fisher1967}, \cite{Kalish1971})
using different methods for the proofs. While Kalish relies
on  Kober's results on integral operators on $L^p-$spaces, Fisher uses the Mikhlin multiplier theorem.
The method of Kalish is therefore close to the one used by Hille-Phillips \cite{HP1957} when $p=2.$
The Riemann-Liouville semigroup $J(z)$ may be viewed as the semigroup of fractional powers of the Volterra operator (for details, see e.g. \cite[Section 8.5]{Haase2006}, in particular Theorem 8.5.8).
Let $\mu\in\mathbb{R}.$ The Hadamard type fractional integration operators of order $\alpha>0$ are given by
\begin{equation}\label{Had-mu-i0}
(\mathcal{J}_\mu^\alpha f)(x)=\frac{1}{\Gamma(\alpha)}\int_0^x(\frac{t}{x})^\mu (\ln(\frac{x}{t}))^{\alpha-1} f(t)\frac{dt}{t}\quad  (x>0),
\end{equation}
and
\begin{equation}\label{Had-mu2-i0}
(\mathcal{I}_\mu^\alpha f)(x)=\frac{1}{\Gamma(\alpha)}\int_x^\infty(\frac{x}{u})^\mu (\ln(\frac{u}{x}))^{\alpha-1} f(u)\frac{du}{u}\quad  (x>0).
\end{equation}

The fractional integral \eqref{Had-mu-i0}  (in the case $\mu=0$)   appears first in Hadamard's paper \cite{Had1892} dedicated to the study of fine properties of  functions which are representable by power series.  It is obtained as a modification of the Riemann-Liouville fractional integral. Our approach to the Hadamard fractional derivative seems to be new in that it involves the abstract theory of fractional powers of semigroup generators while the classical approach consists in modifying the Riemann-Liouville fractional integral (see for example \cite[Section 18.3]{SKM1993}). In the case of the Riemann-Liouville semigroup, the connection  with fractional powers of the generator of the translation semigroup is well documented (see e.g. \cite{ABHN2011, Haase2006, SKM1993}).

We prove that in the spaces $X_c$ (see Section 3 below), and for $1<p<\infty$, these semigroups with parameter $\alpha$ are holomorphic of angle $\frac{\pi}{2}$ and for appropriate values of $c,\, \mu$ they admit boundary values on the imaginary axis. We derive the above representations in a way similar to the Riemann-Liouville case described earlier. From this, the semigroup property which was proved in the papers \cite{BKT02,  BKT02b, Kil01} without appeal to semigroup theory (relying instead on direct computation and the use of the Mellin transform), is obtained  as a consequence. More specifically, our proofs are based on the following two operator semigroups
\begin{align*}
(T_1(t)f)(x)&=f(xe^t),\quad   x>0, \, t>0,\\
(T_2(t)f)(x)&=f(xe^{-t}),\quad  0<x<a,\, \, t>0,
\end{align*}
acting on appropriate weighted $L^p$ spaces, where $1\le p<\infty$ and $a\in (0,\, \infty].$ More precisely, we consider the semigroups $(e^{-\mu t}T_j(t)),\, 1=1,2.$ Such semigroups have been studied extensively in connection with spectral theory and asymptotic behavior the Black-Scholes equation $\displaystyle u_t=x^2 u_{xx}+x u_x$ of financial mathematics (see \cite{AdP2002}). This is a degenerate parabolic equation.  The semigroup $(T_1)$ is also important in spectral theory where it is used to provide counterexamples in various contexts (see e.g. \cite[Chapter 5]{ABHN2011},  \cite{AdP2002}). We are able to recover some results of Boyd \cite{Boyd1968} on the powers of the Ces\`aro operator.

In recent years, fractional calculus has gained increasing interest due to its suitability in modeling several phenomena (deterministic or stochastic)
in science and engineering,  most notably phenomena with memory effects such as anomalous diffusion, fractional Brownian motion and problems in material science,
to name a few.  Some information as well references on these topics can be found in \cite{Ma97},  \cite{Pod99}, \cite{Pr93} and \cite{SKM1993}.

The right nilpotent translation semigroup $S$
is  of contraction operators on $C[0,\, 1]$ whose maximal subspace of strong continuity is $C_0[0,\, 1]:=\{f\in C[0,\, 1],\, f(0)=0\}$. The Lipschitz space  \[{\rm Lip}_0^\alpha[0,\ 1]=C_0[0,\,1] \cap {\rm Lip}^\alpha[0,\ 1]= \left\{ f\in C_0[0,\,1],\, \sup_{t\ne s}\frac{\vert f(t)-f(s)\vert}{\vert t-s\vert^\alpha}<\infty\right\}\]  where $0<\alpha<1$ is a subspace of $C_0[0,\, 1]$ invariant under the  semigroup $S$ but strong continuity fails as well. The maximal subspace of ${\rm Lip}_0^\alpha[0,\ 1]$ on which strong continuity holds is the little H\"older space $\rm{lip}_0^\alpha[0,\ 1].$
A more concrete description of the little H\"older space $\rm{lip}_0^\alpha[0,\ 1]$  is the following:
\begin{equation}\label{Lip-delta}
{\rm lip}_0^\alpha[0,\ 1]=\left\{ f\in C_0[0,\, 1],\, \lim_{\delta\to 0}\sup_{0<\vert t-s\vert\le\delta}\frac{\vert f(t)-f(s)\vert}{\vert t-s\vert^\alpha}=0 \right\}.
\end{equation}

The lack of strong continuity in the space $\rm{Lip}_0^\alpha[0,\ 1]$ is no accident.  Indeed, by a theorem of Ciesielski \cite{Cies1960} (see also \cite[Section 2.7]{BB1967}), for each $\alpha\in (0,\, 1),$  the space $\rm{Lip}_0^\alpha[0,\ 1]$ is isomorphic to the space $l^\infty(\mathbb{N})$ of bounded sequences. In spaces of this type (namely Grothendieck spaces with the Dunford-Pettis property), all generators of strongly continuous semigroups are bounded operators (see \cite[Theorem 4.3.18 and Corollary 4.3.19]{ABHN2011}). The typical spaces in this class are the spaces $L^\infty(\Omega,\mu)$ where $(\Omega,\mu)$ is a measure space, in particlular $l^\infty$, the Hardy space $H^\infty(\mathbb{D})$ of bounded holomorphic functions on the open unit disc, as well as the spaces $C(K)$ of continuous function on a compact Hausdorff space when $K$ is extremally disconnected.  On the other hand, another result of Ciesielski from \cite{Cies1960} states that $\rm{lip}_0^\alpha[0,\ 1]$ is isomorphic to the space $c_0$ of complex sequences converging to $0$.

In the present paper, we shall prove that a boundary group exists in the H\"{o}lder spaces  ${\rm lip}_0^\alpha[0,1], 0<\alpha<1.$ The paper is organized as follows. In Section 2, we study boundary values of holomorphic semigroups by analysing existence of boundary values for individual trajectories $T(z)x.$  The semigroups we consider are holomorphic in some sector $\Sigma_\phi:=\{z\in\mathbb{C}\setminus\{0\},\, \vert \arg(z)\vert<\phi\}$
where $\phi\in (0,\,\frac{\pi}{2}]$ is given.  The Coifman-Weiss transference principle is applied to the Hadamard type semigroups in Section 3 to prove the existence of boundary value groups on the imaginary axis.  On the way to this, we are able to recover some results of Boyd \cite{Boyd1968}  on the powers of the Ces\`aro operator.   We prove existence of the boundary for the Riemann-Liouville semigroup  as our main result in Section 4 by direct estimates allowing us to apply Proposition 1.1.   In the final Section 5, we use a spectral approach to study the question of which groups are boundary value groups. In particular, a generalization of the spectral decomposition in \cite{EM1994_bis} will be established in the more general spectral situation. Precisely, we obtain a direct suitable space decomposition when the spectrum of the generator is assumed to be union of two  connected components which  belong disjointly to $\mathbb C_+$ and $\mathbb C_-$.

\section{Boundary values of a holomorphic semigroup}
\setcounter{theorem}{0}
 \setcounter{equation}{0}

We first  make precise what we  mean  by boundary values of
holomorphic semigroups. Let $A$ be a linear operator in a complex Banach
space $E.$ For $\theta\in (0,\pi]$ let $\Sigma_\theta$ be the sector in the complex plane:
$$\Sigma_\theta :=\{z\in\mathbb{C};\ z\neq 0, \ \vert {\rm
arg}(z)\vert <\theta\}.$$ We recall that $A$ generates a
\textit{bounded holomorphic semigroup of angle} $\phi\in(0,\pi/2]$
if $A$ generates a $C_0-$semigroup $T$ which has a bounded
extension  to each subsector  $\Sigma_{\phi'}$ (where $0<\phi'<\phi$) of the sector $\Sigma_\phi.$
The bound will in general depend on $\phi'.$ Then by analytic continuation, the holomorphic extension, still
denoted by $T$, satisfies the semigroup property: $T(z+z')=T(z)T(z')$ for $z,z'\in
\Sigma_\phi$  and $\lim_{z\rightarrow 0, z\in\Sigma_\phi}T(z)=I$
in the strong operator topology.    The operator  $A$  generates a \textit{holomorphic
semigroup $T$ of angle $\phi$} if for all $\theta \in (0,\phi)$
there exists $\omega\geq 0$ such that $A-\omega I$ generates a
bounded holomorphic semigroup $T_\theta$ of angle $\theta$ (and
then $T(z) = e^{\omega z}T_\theta(z), \ z\in\Sigma_\theta).$ Then
for $\theta \in (-\phi,\phi),\ (T(te^{i\theta}))_{t\geq 0}$ is a
$C_0-$semigroup and its generator is $e^{i\theta}A.$
Conversely, suppose $A$ generates a $C_0-$semigroup. Then,  if for each $\theta \in (-\phi,\phi) $ where $\phi\in (0,\,\frac{\pi}{2}]$, the operator $e^{i\theta}A$
is the generator of a $C_0-$semigroup of contractions, then and only then, $A$ generates a holomorphic semigroup of contractions of angle $\phi$ (see \cite[Theorem 1.54]{Ka2010} where this and similar results are presented).

Due to their importance in the area of partial differential equations (more precisely, those of parabolic type) and in the theory of stochastic processes, holomorphic semigroups have been extensively studied from the early days of the theory of strongly continuous semigroups.  We refer to   \cite{ABHN2011, BB1967, EN2000, HP1957, Pazy83} for more information on the subject. Throughout this paper  we denote respectively by $D(A)$, $\sigma(A)$ and
$\rho(A)$ the domain, the spectrum and the resolvent set of $A.$
In this first section we are interested in the boundary values of
holomorphic semigroups. In this context, we reformulate Proposition 1.1 as:

\begin{proposition}\label{Proposition 2.1} Let $A$ be the generator of a holomorphic
$C_0-$semigroup $T$  of angle $\phi \in(0,\pi/2].$ Then $T(te^{\pm
i\phi})x:=\lim_{\epsilon \rightarrow 0}T(\epsilon +te^{\pm
i\phi})x $ exists uniformly in $t\in[0, 1]$ for all $x\in E$ if
and only if
$$ \sup_{z\in \Sigma_\phi \cap D_{\pm}} \Vert T(z)\Vert<\infty.$$
\end{proposition}

Proposition \ref{Proposition 2.1} is an obvious combination of \cite[Proposition
1.1, Proposition 1.2]{AEH1997}. In fact, in this case, the operators  $(T(te^{\pm
i\phi}))_{0\leq t\leq 1}$ can then be canonically extended to
the half line $[0,\infty)$ as $C_0-$semigroups. We call these
semigroups $(T(te^{\pm i\phi}))_{t\geq 0}$ \textit{the boundary
values of $T.$}  \\
When we deal with parabolic partial differential equations in $L^p$ spaces, for instance,
the holomorphic semigroup $T$ may not be locally bounded on
$\Sigma_\phi $ but the trajectories for smooth initial data may be
locally bounded. In order to assign   also
boundary values, to these semigroups,  we consider for every holomorphic
semigroup $T$ of angle $\phi$  the $T-$invariant  subspaces shortly denoted$E^T$ and $E_T$: \\
$$E^T:=\{ x\in E;\ \lim_{\epsilon \rightarrow
0}T(\epsilon +te^{ i\phi})x \hbox{ converges uniformly for } t
\hbox{ in compact subsets of } [0,+\infty) \},$$
$$E_T:=\{ x\in E;\ \lim_{\epsilon \rightarrow 0}T(\epsilon +te^{ -i\phi})x \hbox{
converges uniformly for } t \hbox{ in compact subsets of }
[0,+\infty) \},$$ where the convergence   is to  be understood  in
$E$. As a consequence of the above uniform convergence we also have
$x\in  E^T$ (respectively $x\in  E_T$ )  if and only if $T(\epsilon+z)x$ converges
uniformly for $z$ in the compact subsets of $\Sigma_\phi^+:=\{z\in
\Sigma_\phi;\   \hbox{Im}(z)\geq 0 \}$ (respectively  $\Sigma_\phi^-:=\{z\in
\Sigma_\phi;\   \hbox{Im}(z)\le 0 \}$).
We then set $T(te^{\pm  i\phi})x:=\lim_{\epsilon \rightarrow
0}T(\epsilon +te^{ \pm i\phi})x$ for all $x\in  E^T \cap E_T$.

\begin{proposition} Let $A$ be the generator of a holomorphic
$C_0-$semigroup $T$  of angle $\phi \in(0,\pi/2]$ and
$E^T, E_T$ be the subspaces defined above. Then:\\
i) $E^T \cap E_T$ is dense in $E.$\\
ii) For all $x\in E^T$, $z, z'\in \overline{\Sigma}_{\phi}^+ $ (respectively $x\in E_T$, $z, z'\in \overline{\Sigma}_{\phi}^- $) 
 we have $T(z')T(z)x =T(z+z')x $ and
$\overline{\Sigma}_\phi^+\ni z\mapsto T(z)x$ (respectively $\overline{\Sigma}_\phi^-\ni z\mapsto T(z)x$ ) is continuous with
values in $E.$
\end{proposition}

\begin{proof}
(i) Let $x\in E$. For all $t_0>0$ since $T$ is strongly uniformly
continuous on the compact subsets of the sub-sectors
$\Sigma_{\theta },\ 0<\theta< \phi$  we have $T(t_0)x\in E^T
\cap E_T $ and $T(t_0)x\rightarrow x$ as
$t_0\rightarrow 0.$ The claim follows.\\
(ii) Let $z\in \overline{\Sigma}_\phi^+ $, $z'\in {\Sigma}_{\phi}$
and $x\in E^T.$  We have
 $T(z)x =\lim\limits_{\epsilon \rightarrow 0}T(\epsilon +z)x$ uniformly for $z$ in the compact subsets of $\Sigma_\phi^+$.
 Then
\begin{align*}T(z')T(z)x&=\lim\limits_{\epsilon \rightarrow
0}T(z')T(\epsilon+z)x\\&=\lim\limits_{\epsilon \rightarrow
0}T(\epsilon+z+z')x =T(z+z')x
\end{align*} since $z+z'\in {\Sigma}_\phi.$ The
 continuity of $z\mapsto T(z)x$ on $\overline{\Sigma}_\phi^+$
 follows from the uniform convergence. The same argument prevails for $x\in E_T$ and $z\in \overline{\Sigma}_\phi^-$.
\end{proof}

 The above proof uses the fact that $\bigcup_{t>0}T(t)(E)$ is dense in $E$, which follows from the $\mathcal C_0-$ property.
 We mention that for general holomorphic semigroups, more
 is true: $\bigcap_{t>0}T(t)(E)$ is dense in $E$( see e.g. \cite{CK2000}).

\begin{example}
Let $E=l^2(\mathbb{N})$ and $(A,D(A))$ be given by 
\begin{align*}
A(x_n):&=\big((-n+in+i\ln(n+1))x_n\big)\\
D(A):&=\big\{(x_n)\in l^2(\mathbb{N})\mid \ (nx_n)\in l^2(\mathbb{N})\big\}.
\end{align*}
  Then $A$ with domain $D(A)$ generates a holomorphic $C_0-$semigroup $T$ of angle $\pi/4$ given
by \[T(z)(x_n):=\big(\exp({(-n+in+i\ln(n+1))z})x_n\big)\] for all $(x_n)\in l^2(\mathbb{N})$ and $z\in\Sigma_{\pi /4}.$ Moreover, we have
$E_+^T=E\ $ and $E_-^T=\{ (x_n)\in l^2(\mathbb{N}), (n^\alpha x_n)\in l^2(\mathbb{N}) \hbox{ for all } \alpha
 >0\}.$ In fact,  for all  $z\in \Sigma_{\frac{\pi}{4}}$ and all integer $n$ we have
\begin{align*}
    \|T(z)(x_n)\|=&\| \exp{((-n+in+i\ln(n+1))z})x_n\| \\
    =& \|\exp{((-\sqrt{2} n e^{-\frac{i\pi}{4}}+i\ln(n+1))z})x_n\|
\end{align*}

and for the upper boundary values $z=te^{i\frac{\pi}{4}}, \ t\ge 0,$ we obtain 
\begin{align*}
    \|T(z)(x_n)\|=&\|\exp((-n+in+i\ln(n+1))te^{i\frac{\pi}{4}})x_n\|  \\
    =&\|\exp{((-n\sqrt{2}e^{-i\frac{\pi}{4}}+i\ln(n+1))te^{i\frac{\pi}{4}}})x_n\|  \\
    =& \|\exp(-n\sqrt{2}t-\frac{\ln(n+1)}{\sqrt{2}}t)x_n\|.
   \end{align*}
From this it follows  that $T(te^{i\frac{\pi}{4}})$ is well defined for all $t\ge 0$ and all $(x_n)\in E,$ which means that $E^T=E.$ \\
    Let us now examine the  behavior of the semigroup for lower boundary values $z=te^{-i\frac{\pi}{4}}, t\geq 0.$ For such values $z$ we  obtain 
\begin{align*}
    \|T(z)(x_n)\|=&\|\exp((-n+in+i\ln(n+1))te^{-i\frac{\pi}{4}})x_n\|  \\
    = & \|\exp((ni\sqrt{2}t) +e^{t\frac{\ln(n+1)}{\sqrt{2}}})x_n\|  \\
     = & \| e^{t\frac{\ln(n+1)}{\sqrt{2}}}x_n\| \\
     = & \left(\displaystyle\Sigma_{n\ge 0} (n+1)^{2t/\sqrt{2}}|x_n|^2\right)^{1/2}.
  \end{align*}

In order to ensure the existence of   boundary values for semigroup $T(z)_{\rm Re(z)\ge 0}$, one must have $(n^{\alpha}x_n)_n\in  l^2(\N)$ for all
$\alpha\ge 0$ and all $(x_n) \in l^2(\N).$ This is obviously not the case for an arbitrary  $x$ in $ l^2(\N)$.  However, we  conclude
that \[E_T=\{(x_n)_n\in  l^2(\N) \ \mbox{such that}\ (n^{\alpha}x_n)_n\in  l^2(\N), \ \forall \alpha>0\}.\]
This is a dense subspace of $l^2(\N)$ as it contains the finitely supported sequences.
\end{example}

This example shows that the semigroup lives until a maturity which depends on the regularity of
the initial data. So, for more regular positions (e.g $x\in D^{\infty}=\cap_{n\ge 0}D(A^n)$)
such as $x=(e^{-n})_{n\ge 0}$, the limit $\lim_{\epsilon\rightarrow 0}T(\epsilon+te^{-i\frac{\pi}{4}})$
exists uniformly for all $t\ge 0$ in compact subsets of $\R^+.$


Now,  $t\mapsto T(te^{+ i\phi})$ and $t\mapsto T(te^{- i\phi})$ seen as functions on
$[0,+\infty)$ with values respectively  in the spaces $ {\mathcal L}(E^T,
E)$ and $ {\mathcal L}(E_T,
E),$ where $E^T$ and $E_T$ are equipped with an adequate topology, are
called \textit{the boundary values of $T$}. The cases when
$E^T=E$ (resp. $E_T=E$) or $D(A^n)\subset E^T$ (resp. $D(A^n)\subset E_T$) for some integer $n\geq 1$
are of particular interest. The first case is characterized in
Proposition 1.1. The second one, which includes many partial differential equations, is the
subject of the following proposition.

\begin{proposition} Let $A$ be the generator of a holomorphic
$C_0-$semigroup $T$  of angle $\phi \in(0,\pi/2]$ and $n\geq 1 $
be an integer. Then $D(A^n)\subset  E^T$ (resp. $D(A^n)\subset  E_T$)  if and only if $ \Vert
T(z)\Vert_{{\mathcal L}([D(A^n)],E)}$ is locally bounded on
$\Sigma_\phi^+$ (resp.$\Sigma_\phi^-$).
\end{proposition}

\begin{proof}
We prove the claim for $E^T$ and $\Sigma_\phi^+$. Analogously, the same argument  remains true for $E_T$ and $\Sigma_\phi^-.$ Assume that $n\geq 1$ and $D(A^n)\subset E^T$ and let
$\lambda_0\in \rho(A)$.  Then $\lim_{\epsilon \rightarrow
0}T(\epsilon+z)(\lambda_0-A)^{-n}x$ converges for all $x\in E$
uniformly for $z$ in compact subsets of
$ \overline{\Sigma}_\phi^+.$ It follows, by the uniform bounded principle   $\Vert
T(\epsilon+z)(\lambda_0-A)^{-n}\Vert$  is locally bounded
in $ \overline{\Sigma}_\phi^+$ .\\
Conversely, if  $ T(z)(\lambda_0-A)^{-n}$ is locally bounded in
$\Sigma_\phi\cap \{ z;\ \pm \hbox{ Im }(z)\geq 0\},$ since
$E_\pm^T$ is dense it follows that $\lim_{\epsilon \rightarrow
0}(\lambda_0-A)^{-n}T(\epsilon+t\exp(+ i
\phi))x=(\lambda_0-A)^{-n}\lim_{\epsilon \rightarrow
0}T(\epsilon+t\exp(+ i\phi))x$ exists for all $x\in E$ uniformly
for $t$ in compact subsets of $[0,+\infty).$

\end{proof}

In practice the estimate on  $ \Vert
T(z)\Vert_{{\mathcal L}([D(A^n)],E)}$ is not easy to establish and one prefers to estimate $ \Vert
T(z)\Vert_{\mathcal L(E)}. $ The following proposition gives a sufficient condition on
 $ \Vert T(z)\Vert_{\mathcal L(E)} $ to guarantee $D(A^n)\subset E^T\cap E_T.$

\begin{proposition}\label{pro_scda} Let $A$ be the generator of a holomorphic
$C_0-$semigroup $T$  of angle $\phi \in(0,\pi/2]$ and $n\geq 1 $
be an integer.   Assume that there exists some $\gamma \in [0,n)$
such that the function $z\mapsto (\vert z \vert d(z,\partial \Sigma_{ \phi}))^\gamma  \Vert
T(z)\Vert_{{\mathcal L}(E)} $ is locally bounded on $
\Sigma_\phi.$  Then $D(A^n)\subset E^T\cap E_T.$
\end{proposition}

 Observe first that for all $z$ in the sector $\Sigma_{\phi}$, one has $d(z,\partial \Sigma_{ \phi})=\vert\sin(\phi
-\vert{\rm  arg }(z)\vert)\vert.$ This identity will be used in the following proof.
\begin{proof} Without loss of generality we can assume that $0\in
\rho(A).$ Let $n\geq 1,\ x\in E$ and $z\in \Sigma_{ \phi}^+ $. By
the Taylor formula we have 
 $$T(z)A^{-n}x=A^{-n}x+zA^{-n+1}x+...+\frac{z^{n-1}A^{-1}x}{(n-1)!}+
 \frac{1}{(n-1)!}\int_0^z (z-\zeta)^{n-1}T(\zeta)xd\zeta.$$
 It suffices then to  show that $z\mapsto
 \int_0^z\frac{(z-\zeta)^{n-1}}{(n-1)!}T(\zeta)xd\zeta$ is locally
 bounded on $\Sigma_\phi^+.$ Let us denote $(\vert z\vert,z)$ the circular path $\vert z\vert e^{i\theta},\ 0\le \theta \le \arg(z).$ 
 Writing
  $$\int_0^z(z-\zeta)^{n-1}T(\zeta)xd\zeta=\int_0^{\vert
 z\vert}(z-\zeta)^{n-1}T(\zeta)xd\zeta+\int_{(\vert z\vert,z)}(z-\zeta)^{n-1}T(\zeta)xd\zeta,
$$
 and  since $(T(t))_{t\geq 0}$ is
 strongly continuous, we need only to show the local boundedness
 of the second integral. Let $K\subset \Sigma_\phi^+$ be a compact and $c_K>0$ such that
\[(\vert z \vert \vert\sin(\phi -\vert{\rm  arg
}(z)\vert)\vert)^\gamma  \Vert T(z)\Vert_{{\mathcal L}(E)}\leq c_K
\] for all $z\in K.$ A
  direct calculation gives

  \begin{align*}
\Big\Vert\int_{(\vert
    z\vert,z)}\frac{(z-\zeta)^{n-1}}{(n-1)!}T(\zeta)xd\zeta\Big\Vert \leq& \frac{1}{(n-1)!}\int_0^\phi \vert\sin(\phi
-\vert{\rm  arg }(\zeta)\vert)\vert^\gamma )^{n-1-\gamma} \Vert
T(\zeta)x\Vert d\zeta  \\
     \leq & \frac{c_K}{(n-1)!}\vert z\vert^{n-1-\gamma}\int_0^\phi
    \sin^{n-1-\gamma}(\theta)d\theta.
\end{align*}
When $z\in \Sigma_\phi^-$, one may consider mutatis mutandis the circular path $\vert z\vert e^{i\theta},\  \,  \arg(z)\le \theta\le 0$ and do the same calculation.     
\end{proof}

\begin{remark}
     Assuming in Proposition \ref{pro_scda}
  only   $$\sup_{z\in \Sigma_\phi \cap D_\pm}
(\vert z \vert d(z,\partial \Sigma))^\gamma  \Vert T(z)\Vert<\infty $$ then we obtain
analogously  for $x\in D(A^{np}) $  that
$\lim_{\epsilon\rightarrow 0}T(\epsilon +z)x$ converges uniformly
for $z\in \Sigma_\phi^\pm, \ \vert z\vert\leq p$ for all integer
$p\geq 1.$ In other words,  to guarantee the existence of the
boundary value $T(te^{i\pm\phi})x$ for  more time ($0\leq t\leq p$)
we need more regularity on the initial data (compare with \cite{AEK94}).

\end{remark}

\section{Hadamard-type fractional integrals}
\setcounter{theorem}{0}
 \setcounter{equation}{0}

In the articles   \cite{BKT02, BKT02b},  \cite{Kil01}, the Hadamard fractional integral was considered along with a generalization and the semigroup property was established. The Hadamard fractional integral first appeared in the paper \cite{Had1892} on the study of functions presented by power series.  This fractional integral has been studied extensively in the monograph \cite{SKM1993} along with the associated concept of fractional derivative. Here, we prove the semigroup property and obtain that the semigroups involved  are holomorphic with angle $\pi/2$ on  a class of weighted $L^p$ spaces. Moreover, we show that these semigroups admit a boundary value group when $1<p<\infty.$ We shall consider two of the operator families studied in
 \cite{Kil01,BKT02}. The other families can be treated with the methods of the present section. Our approach uses abstract semigroup theory in contrast with \cite{BKT02, BKT02b,  Kil01} where the authors proceed with direct computation.
 Moreover, the above cited papers  do not consider complex parameters in the semigroups. We shall obtain as a consequence  the representation of the semigroup powers of the Ces\`aro averaging operator.

The (generalized) Hadamard fractional integral of a function $f$  is defined as follows (see  \cite{Kil01, BKT02, BKT02b}) 
\begin{equation}\label{Had-mu}
(\mathcal{J}_\mu^\alpha f)(x)=\frac{1}{\Gamma(\alpha)}\int_0^x(\frac{t}{x})^\mu (\ln(\frac{x}{t}))^{\alpha-1} f(t)\frac{dt}{t}\quad  (x>0),
\end{equation}
where $\alpha>0$ and $ \mu\in\mathbb{R}.$ The original definition  corresponds to $\mu=0$ and is discussed in  \cite[Chapter 4, Section 18.3]{SKM1993}.
As in  \cite{Kil01, BKT02, BKT02b}, we consider the Banach spaces
 \[X^p_c=\{ f: (0,\, 1)\longrightarrow\mathbb{C} , f \text{\,\, measurable\,\, and \,} \Vert f\Vert_{c,p}=\left(\int_0^1 \vert  x^c f(x)\vert^p\frac{dx}{x}\right)^{1/p}<\infty\}.\]
The space  $X^p_c$ is a Banach space and if $c=\frac{1}{p},$ it coincides with $L^p(0,\, 1).$ We note that if $\alpha= \mu=1$ then
$(\mathcal{J}_1^1 f)(x)=\frac{1}{x}\int_0^x  f(t) {dt} ,\, x>0,$ which is the Ces\`aro operator, so that   $L^p$ boundedness of $\mathcal{J}_1^1$ is to be compared to Hardy's inequality. It follows that the semigroup $(\mathcal{J}_1^\mu)$ represents the fractional powers of the Ces\`aro operator.

We shall prove the following result on the boundary values of the Hadamard -type fractional integral.

\begin{theorem}\label{Hada1} Let $1\le p<\infty$ and $\mu,\, c\in\mathbb{R}$ with $\mu> c.$ Then
 the family of operators $(\mathcal{J}_\mu^\alpha )_{\alpha>0}$ acting on the space $X^p_c$ forms a strongly continuous semigroup which has an analytic extension to the right half-plane
 $ \C_+=\{\alpha\in\C\mid  \Re (\alpha)>0\}.$ Moreover,  the semigroup  $(\mathcal{J}_\mu^\alpha )_{\alpha>0}$    has a boundary $C_0-$group. More precisely,   $(\mathcal{J}^{is}_\mu)_{s\in\R},$  given by
 \begin{equation}
 (\mathcal{J}_\mu^{is} f)=\lim_{\sigma\to 0^+}(\mathcal{J}_\mu^{\sigma+is} f),\ \forall f\in X^p_c
 \end{equation}
  and $(\mathcal{J}_\mu^{is})_{\mu \in \R}$ forms a $C_0-$group, provided $1<p<\infty$.
\end{theorem}

\begin{proof}
 We consider the operator family
 \begin{equation}\label{Had-mu}
(T(t) f)(x)=  e^{-\mu t}f(e^{-t}x) ,\, x\in (0,\, 1),\, t>0.
\end{equation}
Then it is readily verified that $T=(T(t))_{t\geq 0}$ is a strongly continuous group on the space $ X_c^p$ defined above. The infinitesimal generator of $T$ is the operator
$\displaystyle A=x\frac{d}{dx}-\mu I.$ If $\mu> c$ then $T$ is exponentially stable. In fact,
\begin{eqnarray*}
\Vert T(t)f\Vert_{X_p^c}^p&=&\int_0^1  e^{-\mu p t}\vert x^c f(e^{-t} x)\vert^p \frac{dx}{x}\cr
                    &=&e^{-p(\mu -c) t}\int_0^{e^{-t}}  \vert u^c f(u)\vert^p \frac{du}{u}
                    \cr&\le &e^{-p(\mu -c) t}\Vert f\Vert_{X_p^c}^p.
\end{eqnarray*}
The fractional powers $A^{-\alpha}$ for $\alpha>0$ are given by
the well-known formula (see e.g. \cite[p. 167, Formula (3.56)]{ABHN2011},\, \cite[Proposition 11.1]{Kom1966}) or
\cite[Proposition 3.3.5 and Corollary 3.3.6]{Haase2006}).

\begin{equation}\label{powers-2}
((-A)^{-\alpha}f)(x)=\frac{1}{\Gamma(\alpha)}\int_0^\infty t^{\alpha-1}(T(t)f)(x)dt,\, f\in X.
\end{equation}
By a change of variable in the integral, we have for $f\in X:$
\begin{eqnarray*}
((-A)^{-\alpha}f)(x)&=&\frac{1}{\Gamma(\alpha)}\int_0^\infty t^{\alpha-1}(T(t)f)(x)dt\\
&=&\frac{1}{\Gamma(z)}\int_0^\infty t^{z-1}e^{-\mu t}f(e^{-t}x)dt\\
&=&\frac{1}{\Gamma(\alpha)}\int_0^x(\frac{t}{x})^\mu (\ln(\frac{x}{t}))^{\alpha-1} f(t)\frac{dt}{t}\, \\
\end{eqnarray*}
which is \eqref{Had-mu}. From this representation, the semigroup property for the family  $(\mathcal{J}_\mu^\alpha )_{\alpha>0}$ follows by the general theory of fractional powers of operators (see e.g. \cite[Proposition 3.2.3]{Haase2006}).
 Analyticity also follows from the general theory.

  In order to obtain the last assertion, we note that $(T(t))$ is a strongly continuous semigroup of positive contraction operators on the space $X=L^p(0,\, 1),\, 1<p<\infty.$
 The conclusion is obtained by application of the Coifman-Weiss transference principle (see \cite[Theorem 3.9.5]{ABHN2011}, \cite{C-W76}).
\end{proof}

Let us consider the particular case where $\mu=1$.
\begin{corollary}\label{Corollary Cesaro} 
Assume that $1<p<\infty.$ Then the the family
\begin{eqnarray*}
T(z)f(x)=((-A)^{-z}f)(x)=\frac{1}{x\Gamma(z)}\int_0^x  (\ln(\frac{x}{t}))^{z-1} f(t) {dt},\, x>0,\, Re(z)>0
\end{eqnarray*}
forms a holomorphic semigroup of angle $\frac{\pi}{2}$ in the space $L^p(0,\, 1).$ This semigroup admits a boundary value group on the imaginary axis.
\end{corollary}
 Another consequence of the above representation is the explicit description of the powers of the averaging operator
of the \textit{Ces\`aro operator} $C,$ 
 \[(Cf)(x):=\frac{1}{x}\int_0^x  f(s){ds},\, f\in L^p(0,\, 1).\] 
Clearly $C=T(1).$ The strong continuity of $C$ in $L^p(0,\, 1), \, 1<p<\infty$ yields the Hardy's inequality.   Since  Hardy's inequality does not hold for $p=1$, we see that the condition $\mu>c$ in  \ref{Hada1} is sharp.
\begin{corollary} \label{Cesaro}
For each  $n\in \mathbb{N}$ and $f\in L^p(0,\, 1)$ we have
\begin{equation}\label{Cesaro-n}
(C^nf)(x)=\frac{1}{ x(n-1)!}\int_0^x  (\ln(\frac{x}{t}))^{n-1} f(t) {dt}\quad  .
\end{equation}
\end{corollary}

This is of course a direct consequence of the semigroup property: $C^n=(T(1))^n=T(n).$

We observe that this formula was obtained by D. W. Boyd \cite[Lemma 2]{Boyd1968} who used mathematical induction. He used this result to study the spectral radius of averaging operators. The spectral theory of the Ces\`aro operator (including the discrete version) has been studied in several papers, (see for  example \cite[Section 2]{AdP2002} where the Boyd indices are used in the description of the spectrum in various Banach function spaces).

Boyd obtains the following formula (we consider the case $a=0$ in his formula)
\begin{equation}
(C^nf)(x)=\frac{1}{ (n-1)!}\int_0^1  (\ln(\frac{1}{s}))^{n-1} f(sx) {ds},\, f\in L^p(0,\, 1)
\end{equation}
which is readily obtained from \eqref{Cesaro-n} by a change of variable.  In fact, the semigroup  $(T(t))$ in Corollary \ref{Corollary Cesaro}  can be written as follows
\begin{eqnarray*}
(T(z) f)(x)=\frac{1}{\Gamma(z)}\int_0^1(\sigma)^{\mu-1} (\ln(\frac{1}{\sigma}))^{z-1} f(\sigma x){d\sigma},\,  f\in L^p(0,\, 1) ,\, \Re(z)>0.
\end{eqnarray*}
We observe that the above theorem and its corollaries remain valid if we replace $L^p(0,\, 1)$ with $L^p(0,\, a)$ where $a\in (0,\, \infty]$. We state this below for the case $a=\infty.$ For that we introduce the space
\[X_c^{p} :=\{ f: (0,\, \infty)\longrightarrow\mathbb{C} , f  \text{\,\, measurable\,\, and \,} \Vert f\Vert_{c,p}=\left(\int_0^\infty \vert  x^c f(x)\vert^p\frac{dx}{x}\right)^{1/p}<\infty\}.\]

\begin{theorem}\label{Hada1-ab} Let $1\le p<\infty,\, \mu,\, c\in\mathbb{R}$, with $\mu> c\in\mathbb{R}$ Then
 the family of operators $(\mathcal{J}_\mu^\alpha )_{\alpha>0}$ acting on the space $X_c^p$ forms a strongly continuous semigroup which has an analytic extension to the right half-plane $\C_+.$ Moreover,  the semigroup  $(\mathcal{J}_\mu^\alpha )_{\alpha>0}$    has a boundary $\mathcal C_0-$group on $X_c^p$  denoted   $(\mathcal{J}_\mu^{is})_{s\in\R}$  where
 \begin{equation}
 (\mathcal{J}_\mu^{is} f)(x)=\lim_{\sigma\to 0^+}(\mathcal{J}_\mu^{\sigma+is} f)(x)
 \end{equation}
  provided $1<p<\infty$.
\end{theorem}


For the remainder of this section, we consider a second form of the Hadamard fractional integral operator to which the above construction applies (see again \cite{Kil01, BKT02, BKT02b}  and  \cite[Chapter 4, Section 18.3]{SKM1993} for the case $\mu=0$).
Here we set
\begin{equation}\label{Had-mu2}
(\mathcal{I}_\mu^\alpha f)(x)=\frac{1}{\Gamma(\alpha)}\int_x^\infty(\frac{x}{u})^\mu (\ln(\frac{u}{x}))^{\alpha-1} f(u)\frac{du}{u},\, x>0.
\end{equation}
We obtain the following counterpart of Theorem \ref{Hada1}.

\begin{theorem}\label{Hada2} Let $1\le p<\infty$ and let $\mu\in\mathbb{R}$ such that  $c+\mu >0$.   Then
 the family of operators $(\mathcal{I}_\mu^\alpha )_{\alpha>0}$ acting on $X_c^p$ forms a strongly continuous semigroup which has an analytic extension to the right half-plane $\C_+$.  Moreover, the semigroup has a boundary $C_0$-group on on $X_c^p$ .
\end{theorem}

\begin{proof}
In order to obtain the result, we consider the semigroup:
 \begin{equation}\label{Had-mu-10}
(T(t) f)(x)=  e^{-\mu t}f(e^{t}x) ,\, x>0,\, t>0.
\end{equation}
acting on the space $X_c^p.$ Let denote by $A$ its infinitesimal generator. Then 
\begin{eqnarray*}
\Vert T(t)f\Vert_{X_c^p}^p&=&\int_0^\infty \vert x^ce^{-\mu t}f(xe^t)\vert^p\frac{dx}{x}\cr
                          &=&e^{-\mu p t}\int_0^\infty e^{-cpt}\vert u^c f(u)\vert^p\frac{du}{u}\cr
                          &=&  e^{-p(c+\mu )t}\Vert  f \Vert_{X_c^p}
\end{eqnarray*}
for every $f\in X_c^p$ and  $t>0.$ Since $c+\mu >0,$ we deduce that $T$ is exponentially stable. Next, for $\alpha>0$ we have
\begin{equation}\label{powers-2}
((-A)^{-\alpha}f)(x)=\frac{1}{\Gamma(\alpha)}\int_0^\infty t^{\alpha-1}(T(t)f)(x)dt=\frac{1}{\Gamma(\alpha)}\int_0^\infty t^{\alpha-1}e^{-\mu t}f(xe^t)dt,\, f\in X.
\end{equation}
Again, by  a change of variable in the integral, we have for $f\in X:$
\begin{align*}
((-A)^{-\alpha}f)(x)&=\frac{1}{\Gamma(\alpha)}\int_x^\infty(\frac{x}{u})^\mu (\ln(\frac{u}{x}))^{\alpha-1} f(u)\frac{du}{u},\, x>0,
\\ & =(\mathcal{I}_\mu^\alpha f)(x).
\end{align*}
Now the proof is completed in the same way as  in the proof of Theorem \ref{Hada1}. 
\end{proof}

We remark that we can consider in the above results $ \mu\in\mathbb{C}$, in which case the conditions in Theorem \ref{Hada1} and Theorem \ref{Hada2} become $\rm Re(\mu)>c$ and
$\rm {Re(\mu)}>-c$ respectively.\\

We observe that the theorem holds for $X=L^p(0,\,\infty)$ provided $\mu >-\frac{1}{p}$ because $X_c=L^p(0,\, \infty)$ if $c=\frac{1}{p}.$
  We observe that a consequence of the theorem is the  second Hardy inequality which gives the $L^p-$continuity of $\mathcal{I}_0^1$ where  $(\mathcal{I}_0^1 f)(x)=\int_x^\infty f(u)\frac{du}{u},\, f\in L^p(0, \,\infty)$ for $1\le p<\infty.$

We now check that the above semigroups may be studied in the Lipschitz and H\"older spaces as the Riemann-Liouville semigroup  considered in the next section.
We will do this for the semigroup \eqref{Had-mu} which we further simplify by taking $\mu=0.$ For $h>0, t>0$ and $f\in  C[0,\, 1],$ we have by the Mean Value theorem:

\begin{eqnarray*}
\vert t^{-\alpha} (T(t) f-f)(x)\vert&=& t^{-\alpha}\vert f(e^{-t}x)-f(x)\vert\cr
                &=& t^{-\alpha}\vert  f( x-te^{-\eta t}x)-f(x)\vert\cr
                &=& (xe^{-\eta t})^{\alpha}(te^{-\eta t}x)^{-\alpha}\vert  f( x-te^{-\eta t}x)-f(x)\vert\cr
                 &\le & \sup_{h>0}\{h^{-\alpha}\vert  f( x-h )-f(x)\vert\}\cr
\end{eqnarray*}
It follows that ${\rm Lip}^\alpha[0,\, 1]\subset  F_\alpha$ for $0<\alpha<1$ as well as  ${\rm lip}^\alpha[0,\, 1]\subset X_\alpha.$ Here, $F_\alpha$ and $X_\alpha$ are denote respectively  the {\it{Favard}}  space and {\it{the abstract H\"{o}lder}} space of order $\alpha$ (see\cite[Chapter II]{EN2000}) 
On the other hand, if $0<\alpha\le 1$ and $s>0$, from
\begin{eqnarray*}
t^{-\alpha}(T(t)T(s)f-T(s)f)=t^{-\alpha}T(s)(T(t)f-f)
\end{eqnarray*}
that ${\rm Lip}^\alpha[0,\, 1]$\, \, is   $T(t)-$invariant.

More information  about mapping properties of operators with power logarithmic kernels such as the Hadamard fractional integrals can be found in \cite[Paragraph 21]{SKM1993} where $L^p$ spaces and spaces of H\"older continuous functions are considered. Our aim in this paper is to study the purely imaginary powers from the viewpoint of holomorphic semigroups.

\section{The Riemann-Liouville semigroup}
\setcounter{theorem}{0}
 \setcounter{equation}{0}
We  discuss the
Riemann-Liouville semigroup in connection with the right translation
semigroup $S$ acting on ${\rm  Lip}_\alpha [0,1].$  Later, we shall study
the Riemann-Liouville semigroup in $C_0[0,1]$ and
$\rm{Lip}_0^\alpha[0,1].$  \\

The spaces, which  can be defined in connection with any strongly continuous semigroup,  are called Favard classes (for
$\rm{Lip}_0^\alpha$,
 which are denoted by $F_\alpha$ in \cite{EN2000})   and H\"older spaces  (for $\rm{lip}_0^\alpha$, denoted by $X_\alpha$ in \cite{EN2000}) and are studied systematically in \cite{BB1967} in connection with approximation theory.  In particular, one can find equivalent descriptions of these spaces in terms
of the resolvent of the infinitesimal generator of the semigroup under consideration. They are also related to the continuous interpolation spaces
and have been studied in relation to maximal regularity (for a reference, see \cite[page 155]{EN2000}).

Analogously the Riemann-Liouville semigroup $J:=(J(z))_{{\rm
Re}(z)> 0}$ on $C[0,1]$ given by (\ref{RL_representation}) is  not strongly continuous
since $J(t)1(s)=\frac{s^t}{\Gamma(t+1)}$ does not converge to $1$
as $t\rightarrow 0^+$ (one may consider for example the point evaluations at
$s=\frac{1}{2^n}$ and $t=\frac{1}{n},\  n\in \N$). But, its part in $C_0[0,1]$ defines a
$C_0-$holomorphic semigroup of angle $\pi/2.$ Denoting the
generator of the nilpotent translation semigroup $S$ (acting in
$C_0[0,1]$) by $B$ we have $\rho(B)=\mathbb{C}$ and
$$J(z)f=\frac{1}{\Gamma(z)}\int_0^{+\infty} s^{z-1}S(s)f
ds=(-B)^{-z}f,
 $$
 for all $f\in  C_0[0,1]$ where $(-B)^{-z} $ are the
 complex powers of $B$ as defined for instance in \cite{Pazy83, ABHN2011, EN2000, Kom1966}. Recall that, for all $0<\alpha<1$  the
 domain of the fractional power $(-B)^\alpha $, or equivalently the range of $(-B)^{-\alpha
 }$,  equipped with the norm $\Vert f\Vert_{(-B)^\alpha}:=\Vert
 (-B)^{-\alpha}f\Vert_\infty$ is a Banach space.

\begin{theorem}\label{JblowsinCinf} The family of operators $J=(J(z))_{Re(z)>0}$ defines a strongly continuous holomorphic semigroup
of angle $\pi/2 $ in the space $C_0[0,\, 1].$  Moreover, there exist $C_1,\ C_2 >0$ such that
\begin{equation}\label{JblowsinCinf-E1}C_1\frac{\vert z\vert}{{\rm  Re}(z)}\leq \Vert J(z)\Vert_\infty\leq C_2\frac{\vert z\vert}{{\rm
Re}(z)}\qquad ( {\rm  Re}(z)>0, \vert z\vert \leq 1).
\end{equation}
Furthermore, as operators acting on the space $\Lip_0^\alpha[0,\ 1], \, 0<\alpha<1$, we have
\begin{equation}\label{JblowsinCinf-E2}\|J(z)\|_\alpha\le C\frac{\vert z\vert}{\rm{Re}(z)}\qquad   (\Re(z)>0)
\end{equation}
 for some constant $C>0.$
\end{theorem}

\begin{proof}  The fact that $J$ is a holomorphic semigroup  follows from the general discussion preceding the statement of the theorem. It remains to prove the estimates \eqref{JblowsinCinf-E1} and \eqref{JblowsinCinf-E2}. Set $\D_+:=\{ z\in\C\mid  \Re(z)>0 \text{ and } \vert z\vert \leq 1\}.$

\textit{Step 1:} Let $\epsilon>0,  \ z\in \mathbb{C}^+.$ Consider the two functions $f_\epsilon\in C[0,1]$ and $g_\epsilon\in C_0[0,1]$ respectively defined by  $f\epsilon (t)=(\epsilon+1-t)^{-i\rm{Im}(z)}$  and $g_\epsilon(t):=f_\epsilon(t)-f_\epsilon(0).$ For $\epsilon$ small enough, one may find $t_0\in [0,1]$ such that $|g_{\epsilon}|=2$ and then one has $\|g_{\epsilon}\|_{infty}=2.$ 

\begin{eqnarray*}
    \|J(z)\|_{\infty} &\ge& \frac{1}{2} \lim_{\epsilon\to 0}|(J(z)g_\epsilon)(1)| \\
    &= & \frac{1}{2} \lim_{\epsilon\to 0}\Big|\frac{1}{\Gamma(z)}\int_0^{1}y^{z-1}g_\epsilon(1-y)dy\Big|  \\
    &\ge& \frac{1}{2} \lim_{\epsilon\to 0}\Big|\frac{1}{\Gamma(z)}\int_0^{1}e^{(i{\rm  Im}(z)+{\rm   Re}(z)-1)\ln(y)}e^{ -i{\rm  Im}(z)\ln
    (\epsilon+ y)}dy\Big| \\
    &&\qquad\qquad -\frac{1}{2}\lim_{\epsilon\to 0}\Big\vert \frac{1}{\Gamma(z)}\int_0^1 y^{z-1}(\epsilon+1)^{-i{\rm Im}(z)}dy\Big\vert\\
   \label{terme manque}  & \ge &   \frac{1}{2}\lim_{\epsilon\to 0}\Big|\frac{1}{\Gamma(z)}\int_0^{1}e^{({\rm   Re}(z)-1)\ln(y)}\ln(\epsilon+y)dy\Big|-
 \frac{1}{2} \frac{1}{\vert \Gamma(z)||z|}.  
\end{eqnarray*}
To obtain the first part in inequality in \eqref{JblowsinCinf-E1}, we distinguish two cases: $|\arg(z)|\ge \frac{\pi}{4}$ and $|\arg(z)|< \frac{\pi}{4}$. 
In the latter case, To justify the result on $\mathcal C_0[0,1],$ it suffices to consider  unity approximation sequence $g_n(t)=\frac{t}{n} 1_{[0,\frac{1}{n}]}(t)+ 1_{[\frac{1}{n},1]}(t)$. So

\begin{align*}\|J(z)\|_{\mathcal L(C_0)}&\ge   \|J(z)g_n\|_{\infty}
\\&\ge |\int_{0}^{t}g_n(t-s)s^{z-1}ds|\\
&\ge| \liminf\int_{0}^{t}g_n(t-s)s^{z-1}ds| \\
&  \ge| \int_{0}^{t} \liminf g_n(t-s)s^{z-1}ds|\\
& \ge \frac{1}{|z\Gamma(z)|} \\
& \ge \frac{1}{2\Gamma(z+1)} \frac{|z|}{\rm Re(z)}
\end{align*}
 
 In this case the desired estimate is obvious. 
In the second case (i.e $|\arg(z)|\ge \frac{\pi}{4}$),  using Lebesgue's dominated convergence theorem in the last estimate above we deduce that

\begin{align*}\|J(z)\|_\infty&\ge\frac{1}{2|\Gamma(z)|\Re(z)}- \frac{1}{2\vert \Gamma(z)\vert|z|}
\\&\ge\frac{1}{2|\Gamma(z)|\Re(z)}\big(1-\frac{\Re(z)}{\vert z\vert}\big)\\
&\ge\frac{1}{2|\Gamma(z)|\Re(z)}\big(1-\frac{\sqrt 2}{ 2}\big)\\
&\ge \frac{1}{4|\Gamma(z+1)|}\frac{|z|}{\Re(z)}.
\end{align*}
This last estimate establishes the first part in inequality (\ref{JblowsinCinf-E1}) with $C_1=\frac{1}{4|\Gamma(z+1)|}.$ \\
\textit{Step 2:} For each $f\in C_0[0,1]$ and $z\in \D_+$ we have 
\[\vert \Gamma(z) J(z)f(t)\vert\leq
\int_0^t\vert f(t-s)\vert s^{\hbox{Re}(z)-1}ds\leq \frac{\Vert
f\Vert_\infty}{\hbox{Re}(z)}.\] It follows that the second inequality in \eqref{JblowsinCinf-E1} holds with  $C_2:=\underset{z\in\D_+}{\sup} \frac{1}{\vert\Gamma(z+1)\vert}.$\\

\textit{Step 3:}

Let $f\in \lip_0^{\alpha}[0,\, 1]$ and $z\in \C$ such that  $\Re(z)>0.$ By a direct computation we obtain that 
\begin{align*}
\|J(z)f\|_{\alpha}=&\sup_{t\ne
t'}\Big\vert \frac{ (t-t')^{-\alpha}}{\Gamma(z)}\Big(\int_0^t\frac{f(t-s)-f(t'-s)}{s^{1-z}}ds+\int_t^{t'}\frac{f(t'-s)}{s^{1-z}}
ds\Big)\Big\vert.
\end{align*}
It follows

\begin{align*}
\|J(z)f\|_{\alpha}\le & \sup_{t\ne
    t'}\frac{1}{|\Gamma(z)|}\Big(\int_0^ts^{{\rm  Re}(z)-1}\frac{|f(t-s)-f(t'-s)|}{|t'-t|^{\alpha}}ds\\
&\qquad \qquad \quad +\int_t^{t'}s^{{\rm
Re}(z)-1}\frac{|f(t'-s)-f(0)|}{|t'-s|^{\alpha}} ds\Big)\\
\leq &  \sup_{t\ne
    t'}\frac{\|f\|_\alpha}{|\Gamma(z)|}\Big(\int_0^ts^{{\rm  Re}(z)-1}ds+ \int_t^{t'}s^{{\rm
Re}(z)-1}ds\Big)\\
\le &\frac{2\Vert f\Vert_\alpha}{\vert \Gamma(z)\vert}\int_0^{1}s^{\rm{Re}(z)-1}ds
\\=&\frac{2\Vert f\Vert_\alpha}{\vert \Gamma(z)\vert|\Re(z)|}
\end{align*}
Hence,
\[\|J(z)f\|_{\alpha}\le
\frac{1}{|\Gamma(z)|}\frac{\|f\|_{\alpha}}{{\rm   Re}(z)}\le
\frac{|z|}{|\Gamma(z+1)|}\frac{\|f\|_{\alpha}}{{\rm    Re}(z)}.\]
Thus
$$
\| J(z)\|_{\alpha}\le C  \frac{|z|}{{\rm   Re}(z)}
$$
where $C:=\sup_{z\in {\C_+}}\frac{2}{|\Gamma(z+1)|} $, in particular we
obtain $\|J(\alpha)\|_{\alpha}\le \frac{2}{\Gamma(\alpha+1)}.$ This completes the proof of the theorem using Proposition \ref{Proposition 2.1}.

\end{proof}
Let $G$ be the generator of $J$ in $E=C_0[0,1]$. It follows from
Theorem \ref{JblowsinCinf} that $J$ is not bounded on $D_+$ nor on $D_-$.
Consequently $\pm iG$ does not generate any $C_0-$semigroups.\\
The following property of the semigroup $J$ is needed.

\begin{proposition}\label{PropEN}\cite[Proposition 5.33 p.140]{EN2000} Let $0<\beta <\alpha<1$ and $J$ be the
Riemann-Liouville semigroup in $C_0[0,1]$. Then:\\
(i) $J(\alpha)C_0[0,1]\subset {\rm   Lip}_0^\alpha[0,1] $ and
$\Vert
J(\alpha)f\Vert_\alpha \leq\frac{2}{\Gamma(\alpha+1)}\Vert f\Vert_\infty.$\\
(ii) ${\rm   Lip}^\alpha_0[0,\, 1]\hookrightarrow D((-B)^\beta).$

\end{proposition}


In the next theorem, which is our main result, we show that the Riemann-Liouville semigroup $J$ admits a boundary group  on $\Lip_0^\alpha[0,1].$ This group defines therefore the
operation of fractional integration of imaginary order as a bounded  strongly continuous group in the H\"older spaces.
It will be important to note that in Proposition 2.1, we can replace $D$ by $D_R:=\{z\in \mathbb{C},\, Re(z)>0\, \vert z\vert \le R\}$ where $R$ is any fixed positive real number. This simple fact is seen by observing that if $R>0$ is a given number, and $(T(z))$ is a holomorphic semigroup of angle $\frac{\pi}{2}$  whose infinitesimal generator is $A$, then
\begin{equation*}
\sup_{z\in D_R}\Vert T(z)\Vert_{{\mathcal L}(E)} =\sup_{z\in D}\Vert T(Rz)\Vert_{{\mathcal L}(E)} .
\end{equation*}
It suffices then to note that ${(T(Rt))}_{t\ge 0}$ is a holomorphic semigroup and its generator is $RA.$
\begin{theorem}\label{J_holm_onLip} Let $0<\alpha<1$. Consider
 the Riemann-Liouville semigroup
    $J$ in the space $\Lip^\alpha_0[0,\, 1].$ Then
    \begin{enumerate}
        \item[(i)] The semigroup $J$ is  holomophic of
        angle $\frac{\pi}{2}.$ 
        \item[(ii)]
        In addition,
        ${\sup_{z\in D_\pm}}\Vert J(z)\Vert_{{\mathcal L}(E)} <\infty.$ 
    \end{enumerate}
Consequently the semigroup $J$ admits a boundary group on the imaginary axis.
\end{theorem}
\begin{proof}
The first item (i) is an immediate consequence of Theorem \ref{JblowsinCinf}.
 We shall prove the estimate (ii) with $R=\frac{1-\alpha}{2}$. Let $h=t-t'>0$ where  $ (t,t')\in[0,1]^2$ and $z\in D_R$ Then we have

    \begin{align*}
    \Gamma(z) (J(z)f(t)-&J(z)f(t'))\\&=\int_0^tf(t-s)s^{z-1}ds-\int_0^{t'}f(t'-s)s^{z-1}ds\\
    =&\int_{-h}^{t'}f(t'-s)(s+h)^{z-1}ds+\int_0^{t'} f(t'-s)s^{z-1}ds\\
    =&\int_{-h}^{0}f(t'-s)(s+h)^{z-1}ds+\int_{0}^{t'}f(t'-s)[(s+h)^{z-1}-s^{z-1}]ds\\
    =& \int_{-h}^{0}f(t'-s)(s+h)^{z-1}ds+\int_{0}^{t'}(f(t'-s)-f(t'))[(s+h)^{z-1}-s^{z-1}]ds+\\
    & +\int_{0}^{t'}f(t')[(s+h)^{z-1}-s^{z-1}]ds\\
    =:&I_1+I_2+I_3.\\
    \end{align*}

First, we estimate  the sum of the two  integrals $I_1$ and $I_3.$ We have
by a direct computation
    \begin{align*}
    I_1+I_3  =&   \int_{0}^{h}f(t'-s+h)s^{z-1}ds+f(t')\int_{0}^{t'}[(s+h)^{z-1}-s^{z-1}]ds\\
    =& \int_{0}^{h}f(t-s)s^{z-1}ds+f(t')\int_{0}^{t'}[(s+h)^{z-1}-s^{z-1}]ds\\
    =& \int_{0}^{h}\frac{f(t-s)-f(t)}{s^{\alpha}}s^{\alpha +z-1}ds+\int_{0}^{h}f(t)s^{z-1}ds+f(t')\int_{0}^{t'}[(s+h)^{z-1}-s^{z-1}]ds\\
    =& \int_{0}^{h}\frac{f(t-s)-f(t)}{s^{\alpha}}s^{\alpha +z-1}ds+\frac{f(t)}{z}h^{z}+\frac{f(t)}{z}(t^z-h^z-t'^z)\\
    =& \int_{0}^{h}\frac{f(t-s)-f(t)}{s^{\alpha}}s^{\alpha +z-1}ds+\frac{f(t)-f(t')}{z}h^z+\frac{f(t')}{z}(t^z-t'^z).\\
    =:& K_1+K_2+K_3
    \end{align*}
    It is easy to check that for $\Re(z)>0,$
    \begin{equation}\label{est-k1}
    |K_1|\le \int_0^h  \|f\|_{\alpha}s^{{\rm   Re}(z)+\alpha-1}ds\le \|f\|_{\alpha}\frac{h^{\alpha+{\rm   Re}(z)}}{\alpha+{\rm   Re}(z)}\le \Vert f\Vert_\alpha h^\alpha/\alpha.
   \end{equation}
\noindent  and
    \begin{equation}\label{est-k2}
    |K_2|\le \frac{h^{\alpha+{\rm   Re}(z)}{\|f\|_\alpha}}{|z|}\le \frac{h^\alpha{\|f\|_\alpha}}{{\vert z\vert}}.
    \end{equation}
 \noindent  To estimate the term $K_3$ we discuss two following two cases

 \textit{Case 1:} $t'\le 2h$. Then 
    \begin{align*}
    |\frac{1}{h^{\alpha}}\frac{f(t')}{z}(t^z-t'^z)|\le& \frac{t'^{\alpha}}{h^{\alpha}}\vert t^z-t'^z\vert\frac{\|f\|_{\alpha}}{|z|}\\
    \le& \frac{t'^{\alpha}}{h^{\alpha}}( t^{{\rm  Re}(z)}+t'^{{\rm   Re}(z)})\frac{\|f\|_{\alpha}}{|z|}\\
    \le& \frac{2^{1+\alpha}}{\vert z\vert}{\|f\|_{\alpha}}
    \end{align*}
where we have used  that  $|f(t')|\le t'^{\alpha}\|f\|_{\alpha}$.

  \textit{Case 2:}  $t'>2h.$ Then we have

   \begin{align*}
   \big|\frac{1}{h^{\alpha}}\frac{f(t')}{z}(t^z-t'^z)\big|\le& \frac{t'^{\alpha}}{h^{\alpha}}\frac{\|f\|_{\alpha}}{|z|}\vert t^z-t'^z \vert\\
   \le& \frac{1}{|z|}\frac{t'^\alpha}{h^\alpha}\|f\|_{\alpha}\Big\vert\int_{t'}^t z u^{z-1} du\Big\vert\\
    \le& \|f\|_{\alpha} h^{1-\alpha}t'^\alpha\sup_{t'<u<t}{u^{{\rm   Re}(z)-1}}\\
     \le& \|f\|_{\alpha} h^{1-\alpha} {t'^{\alpha+{\rm   Re}(z)-1}}\\
    \le & \|f\|_{\alpha} h^{1-\alpha}2^{\alpha+{\rm   Re}(z)-1}h^{\alpha+{\rm   Re}(z)-1}\\
     \le & \|f\|_{\alpha} 2^{\alpha+{\rm   Re}(z)-1}h^{{\rm   Re}(z)}\\
    \le &  \|f\|_{\alpha}.
   \end{align*}
 {Thus we conclude that 
\begin{equation}\label{est-k3}
K_3\le \frac{2^{1+\alpha}}{\vert z\vert} h^\alpha\Vert f\Vert_\alpha.
\end{equation}
Combining \eqref{est-k1}, \eqref{est-k2} and \eqref{est-k3} we deduce that 
\begin{equation}\label{est-I_1+I_3}
I_1+I_3\leq h^\alpha\|f\|_\alpha\Big(\frac{1+2^{1+\alpha}}{|z|}+\alpha^{-1}\Big).
\end{equation}}
%

Now we come back to $I_2$
\begin{align*}
I_2=&\vert\int_{0}^{t'}(f(t'-s)-f(t'))[(s+h)^{z-1}-s^{z-1}]ds\vert\\
 \le&
 \|f\|_{\alpha}\vert\int_{0}^{t'}s^{\alpha}|(s+h)^{z-1}-s^{z-1}|ds\vert \\
\le& \|f\|_{\alpha}\int_{0}^{t'}s^{\alpha}h^{{\rm
Re}(z)-1}|(1+\frac{s}{h})^{z-1}-
(\frac{s}{h})^{z-1}|ds\\
\le & \|f\|_{\alpha}\int_{0}^{\frac{t'}{h}}h^{{\rm
Re}(z)+\alpha}s^{\alpha}
|(1+s)^{z-1}-s^{z-1}|ds\\
\le & \|f\|_{\alpha} h^{{\rm   Re}(z)+\alpha}
\int_{0}^{\infty}s^{\alpha+
{\rm   Re}(z)-1}|(1+\frac{1}{s})^{z-1}-1|ds\\
\le & \|f\|_{\alpha} h^{{\rm   Re}(z)+\alpha}\Big( \int_{0}^{2}s^{\alpha+{\rm   Re}(z)-1}|(1+\frac{1}{s})^{z-1}-1|ds\\
&+ \int_{2}^{\infty}s^{\alpha+{\rm   Re}(z)-1}|(1+\frac{1}{s})^{z-1}-1|ds\Big)\\
=:& \|f\|_{\alpha} h^{{\rm   Re}(z)+\alpha}(J_1+J_2).
\end{align*}

\noindent As above, we will treat separately $J_1$ and $J_2$. For $J_1,$ one can verify that if ${\rm   Re}(z)<1,$ we have
\begin{align*}
J_1=&\int_{0}^{2} s^{\alpha+{\rm
Re}(z)-1}|(1+\frac{1}{s})^{z-1}-1|ds\\
\le & \int_{0}^{2} s^{\alpha+{\rm   Re}(z)-1}((1+\frac{1}{s})^{{\rm   Re}(z)-1}+1)ds) \\
\le&   \int_{0}^{2} s^{\alpha+{\rm   Re}(z)-1}(2^{{\rm   Re}(z)-1}+1)ds)\\
\le &  (2^{{\rm   Re}(z)-1}+1)\frac{2^{{\rm   Re}(z)+\alpha}}{{\rm   Re}(z)+\alpha}\\
\le& \frac{2^{1+{\rm   Re}(z)+\alpha}}{{\rm   Re}(z)+\alpha}.
\end{align*}

\noindent We now we estimate $J_2.$ For that we note that  $ \vert (1+\frac{1}{s})^{z-1}-1\vert\le \frac{1}{s}\vert z-1\vert$ for every $s>0$ and $ \Re(z)<1.$ Indeed, if $g(u)=(1+{u})^{z-1},\, u\ge 0$ then $g^\prime(u)=(z-1)(1+{u})^{z-2}.$
By the mean value inequality,
$\vert g(u)-g(0)\vert=\vert (1+{u})^{z-1}-1\vert\le u\sup_{v>0}\vert (z-1)(1+{v})^{z-2}\vert.$ But $\vert (z-1)(1+{v})^{z-2}\vert=\vert (z-1)\vert (1+{v})^{\rm Re(z)-2}\vert\le \vert z-1\vert$
for $\rm Re(z)<1.$ From this, it follows that
\begin{align*}
J_2=\int_{2}^{+\infty} s^{\alpha+{\rm
Re}(z)-1}|(1+\frac{1}{s})^{z-1}-1|ds\le&
 \vert z-1\vert\int_{2}^{+\infty} s^{\alpha+{\rm   Re}(z)-2}ds\\
\le &\frac{\vert z-1\vert}{1-\alpha-\rm{Re}(z)}
\end{align*}
 for all $z\in\C^+$ such that $1-({\rm   Re}(z)+\alpha)>0$
Therefore, we have
\begin{eqnarray*}
I_2\le \Vert f\Vert_\alpha h^{\alpha+\rm{Re}(z)}\Big[\frac{2^{1+\alpha+\Re(z)}}{\Re(z)+\alpha}+\frac{\vert z-1\vert}{1-\alpha-\rm{Re}(z)}\Big]
\end{eqnarray*}
From this, we conclude that for $0<|z|\le \frac{1-\alpha}{2}$ we have
\begin{equation}\label{est-I2}
I_2\le h^\alpha \Vert f\Vert_\alpha (\frac{4}{\alpha}+\frac{2\vert z-1\vert}{1-\alpha}).
\end{equation}

Finally, combining \eqref{est-I_1+I_3} and \eqref{est-I2} we obtain 
$$\|J(z)\|_\alpha\le \frac{1}{\vert\Gamma(z)\vert}(\frac{5}{\alpha}+\frac{2\vert z-1\vert}{1-\alpha})+\frac{1+2^{1+\alpha}}{\vert\Gamma(z+1)\vert}$$
 for $0<|z|\le \frac{1-\alpha}{2}.$ This   completes the proof of the theorem.

\end{proof}

Recall that $S$ is a strongly continuous semigroup on $C_0.$ Since
$C^1_{00}\subset D((-B)^\alpha=J(\alpha)C_0\hookrightarrow {\rm
 Lip}_0^\alpha[0,1]$ we obtain that $D((-B)^\alpha$ is dense in ${\rm
 Lip}_0^\alpha[0,1].$\\
  Seen as  a linear application on $C_0[0,1]$ with values in   $D((-B)^\alpha),$  $J(\alpha)$ is an
 isomorphism. We deduce from
 above  that $J$ induced  also a holomorphic semigroup in
 $[D((-B)^\alpha]$ of angle $\pi/2$ which is not locally bounded. The following diagrams illustrate the foregoing. description.\\

%
%
%
%
%

On one hand, the behavior of $J(z)$ on $C_0[0,1]$ and on $D((-B)^\alpha)$ is the same since it embeds the first space in the second so the following diagram is commutative

\begin{center}

    \begin{tabular}{ccc}
        $C_0[0,\,1]$&$\overset{J(z)}{\longrightarrow} $& $C_0[0,1]$ \\
        $\bigg\downarrow J(\alpha)$& &$\bigg\uparrow J(\alpha)^{-1}$\\
        $D((-B)^\alpha)$&$\overset{J(z)}{\longrightarrow}$ &$D((-B)^\alpha)$\\

\end{tabular}

\end{center}

On the other hand, according to theorem \ref{J_holm_onLip}, $J(z)_z$ is locally bounded on ${\rm  Lip}^\alpha_0[0,\,1]$ but not on $D((-B)^\alpha).$ It yields that the diagram

\begin{center}

    \begin{tabular}{ccc}

        $D((-B)^\alpha)$&$\overset{J(z)}{\longrightarrow}$ &$D((-B)^\alpha)$\\
        $\DownArrow[1.5cm][>=stealth, thick, dashed]  i$& &$\UpArrow[1.5cm][>=stealth, thick, dashed] $\\
        ${\rm   Lip}^\alpha_0[0,\,1]$&$\overset{J(z)}{\longrightarrow} $ &${\rm   Lip}^\alpha_0[0,\,1]$ \\
    \end{tabular}


\end{center}

cannot commute.

It is important to recall that  the norms considered here are the natural ones induced by the considered embeddings. For example, $D((-B)^\beta)$ is endowed with the norm $\|y\|_{D((-B)^\beta}=\|J(\beta)x\|_{\infty}$ where $x$ is the unique element such as $J(\beta)x=y.$ As an immediate consequence we remark that since $J(z)$ is locally bounded in $ {\rm   Lip}^\alpha_0[0,\,
1]$ but it is not in $D((-B)^\alpha)$, then in truth $D((-B)^\alpha)\subsetneq {\rm   Lip}^\alpha_0[0,\,
1]. $
\section{From the boundary  group to the semigroup}
\setcounter{theorem}{0}
 \setcounter{equation}{0}

This section is devoted to the inverse problem on the half plane.
Let $(T(it))_{t\in \mathbb{R}}$ be a $C_0-$group  with generator $iA.$ Is it
the boundary value of some holomorphic semigroup of angle $\pi/2$?\\
 A great work was done before in this direction.
This problem was studied in \cite{EM1994} and \cite{EK1996}  and affirmatively solved under
spectral conditions $(s(A):=\sup\{{\rm   Re}(z); \  z\in
\sigma(A)\}<+\infty)$ on the generator and a growth condition (non
quasi analyticity) on the group. The problem was also studied by
Zsid\'o and Cioranescu \cite{CZ1976} in terms of analytic generators.
\\
Let $iA$ be the generator de the $C_0-$group $(T(it))_{t\in
\mathbb{R}}$ in the $E$ and $w\geq 0$ such that $\sup_{t\in
\mathbb{R}}e^{-w\vert t\vert}\Vert T(it)\Vert<\infty.$ Then
$\sigma(iA)\subset \{z\in\mathbb{C};\ -w\leq {\rm   Re}(z) \leq
w\}$ which means $\sigma(A)=\sigma\subset \{z\in\mathbb{C};\ -w\leq {\rm
Im }(z) \leq w\}.$ In this section, we assume that:

\begin{equation}\label{split_specrum}
\sigma=\sigma_1 \cup \sigma_2 \ \mbox{such that}\ \sigma_1\subset
\{z\in  \C, {\rm   Re}(z)< -\delta \}\  \mbox{and}\
\sigma_2\subset \{z\in  \C, \ {\rm   Re}(z) >+\delta \}
\end{equation}
for some $\delta>0.$ This situation generalizes that one treated in \cite{EM1994_bis} where the spectrum $\sigma$ is assumed belonging entirely in a half plane with a single connected component ($\sigma\subset\mathbb C_-$.) A similar and more general result is given by the following:

\begin{proposition}\label{proposition_splitting}
   Let  $(T(it))_{t\in
\mathbb{R}}$ be a $C_0-$group in $E$  with generator $iA$. Assume
furthermore that $\sigma(A)=\sigma_1\cup \sigma_2$ with
$\sigma_1\subset \{z\in  \C, {\rm   Re}(z)< -\delta \}\ \mbox{ and
}\ \sigma_2\subset \{z\in  \C, \ {\rm   Re}(z) >+\delta \}$. Then
there exist two $T(it)-$invariant and closed subspaces $E_1$ and
$E_2$ of $E$
such that:\\
        {\rm   (i)} $D(A)\subset E_1\oplus E_2,$\\
        {\rm   (ii)}$A_1:=A_{\shortmid_{ E_1}}$ generates a  holomorphic semigroup $T_1$ of
        angle $\frac{\pi}{2}$ in $E_1$ and $\sigma(A_1)=\sigma_1,$\\
        {\rm   (iii))}$-A_2:=-A_{\shortmid_{ E_2}}$ generates a holomorphic semigroup $T_2$ of angle $\frac{\pi}{2}$
        and $\sigma (A_2)=\sigma_2.$
\end{proposition}
	Proposition \ref{proposition_splitting} generalizes   the splitting theorem established in \cite{EM1994_bis}. Indeed, it suffices to choose $\sigma_1=\emptyset$ in order to recover the main theorem 1 therein.\\
 The proof gives an explicit construction of these spaces and semigroups based on a $A-$bounded
 spectral projection.\\
  Let  $(T(it))_{t\in
\mathbb{R}}$ be a $C_0-$group in $E$  with generator $iA$. Then there exists $w\geq 0$ such that $M:= \sup_{t\in
\mathbb{R}}e^{-w\vert t \vert}\Vert T(it)\Vert<\infty .$ In
particular $\sigma_1\subset \{z\in \mathbb{C}; \ -w\leq \hbox{
Im}(z)\leq w, \hbox{Re}(z)\geq \delta\}$ and
$\sigma_2\subset\{z\in \mathbb{C}; \ -w\leq \hbox{ Im}(z)\leq w,
\hbox{Re}(z)\leq -\delta\}.$ Moreover, the resolvent of $\pm iA$
is given by the Laplace transform of $(T(\pm it))_{t\geq 0}$ and
consequently
$$\vert {\rm Re}(\lambda)-w\vert \Vert R(\lambda,\pm iA)\Vert\leq
M \qquad \mbox{for}\quad {\rm Re}(\lambda)\geq w.$$
The latter estimate guarantees in particular the absolute convergence of all integrals involved below.\\
Let $0<r<\delta$ and $\beta \in (\frac{\pi}{2},\pi)$ such that $\sigma_1$ lies in the left side of the
path $\Gamma=(-\infty e^{-i\beta},re^{-i\beta}]\cup
\{re^{i\theta};\ -\beta \leq \theta\leq \beta\} \cup
[re^{i\beta},\infty e^{i\beta}) .$
 For all $t>0$
 define
$$ T_1(t)x:= \frac{1}{2i\pi}\int_{\Gamma}e^{z\lambda}R(\lambda,A)xd\lambda,\
T_2(t)x:=
\frac{1}{2i\pi}\int_{\Gamma}e^{z\lambda}R(\lambda,-A)xd\lambda,$$
then  $T_k(t)\in \mathcal{L}(E)$ is well defined and does not
depend  on the choice of $\beta , r $ for $k=1, 2$. Furthermore,
we have  the following lemma.
\begin{lemma}\label{smgprop} Let $k\in \{1,2\}.$
    The  family $(T_k(t))_{t\ge 0}$ satisfies   the semigroup property .
\end{lemma}

\begin{proof} Let $k=1$ and $t,s>0$. Let $\Gamma'=\Gamma'(r',\beta')$ with $\beta'>\beta$ and $r'<r.$ We have by the resolvent equation and the Cauchy formula:
    \begin{eqnarray*}
    T_1(t)T_1(s)x&=& T_1(t)\frac{1}{2i\pi}\int_{\Gamma'}e^{s\lambda}R(\lambda,A)xd\lambda\\
    &=& \frac{1}{2i\pi}\int_{\Gamma}e^{t\lambda}R(\lambda,A)(\frac{1}{2i\pi}\int_{\Gamma'}e^{s\mu}{R(\mu,A)}d\mu)d\lambda\\
    &=& \frac{1}{2i\pi}\int_{\Gamma}e^{t\lambda}(\frac{1}{2i\pi}\int_{\Gamma'}e^{su}R(\lambda,A)R(u,A)xd\mu)d\lambda\\
    &= & \frac{1}{2i\pi}\int_{\Gamma}{e^{t\lambda}}(\frac{1}{2i\pi}\int_{\Gamma'}e^{su}\frac{-R(\lambda,A)+R(\mu,A)}{\mu-\lambda}xd\mu)d\lambda\\
    &= & \frac{1}{2i\pi} \int_{\Gamma}e^{\lambda(t+s)} {R(\lambda,A)}xd\lambda  \\
    &= & T_1(t+s)x \quad \mbox{for all}\quad x\in E.
    \end{eqnarray*}
Replacing $A$ by $-A$ we obtain also by this calculation
$T_2(t)T_2(s)=T_2(t+s).$

\end{proof}
Observe that for all $s>0$ and all $x\in D(A)$ we have

\begin{eqnarray*}
    T_1(s)x-x&=& \frac{1}{2i\pi}\int_{\Gamma}e^{s\lambda}R(\lambda,A)xd\lambda-x\\
    &=& \frac{1}{2i\pi}\int_{\Gamma}e^{s\lambda}[R(\lambda,A)x-\frac{x}{\lambda}]d\lambda\\
    &=& \frac{1}{2i\pi}\int_{\Gamma}e^{s\lambda}\frac{R(\lambda,A)Ax}{\lambda}xd\lambda.\\
\end{eqnarray*}

 The right hand side in the last equality does not always coincide
with $x$ when $s \rightarrow 0,$  but thanks to the Lebesgue
theorem, one sees that

\begin{equation}\label{limT}
\lim_{s\rightarrow
0}T_1(s)x=\frac{1}{2i\pi}\int_{\Gamma}\frac{R(\lambda,A)Ax}{\lambda}d\lambda+x
,
\end{equation}
for all $x\in D(A).$ It is so allowed to define an unbounded
projection $ P\ \ x\mapsto Px=\frac{1}{2i\pi}\int_{\Gamma}\frac{R(\lambda,A)Ax}{\lambda}d\lambda+x. $

The relation (\ref{limT}) says that

\begin{equation}\label{limT=P}
\lim_{s\rightarrow 0}T_1(s)x=Px.
\end{equation}

The semigroup property and identity (\ref{limT=P}) ensure that $P$
is an unbounded  projector, which  means $P^2x=Px$ for all $x\in
D(A^2).$ Indeed one may write $\lim_{t+s\rightarrow 0}T_1(t+s)x=Px$
where $t$ and $s$ are assumed to be  positive and so converge both
to zero.
 But $\displaystyle\lim_{t+s\rightarrow 0}T_1(t+s)x=\displaystyle\lim_{t\rightarrow 0,s\rightarrow 0}T_1(t)T_1(s)x=
 \displaystyle\lim_{t\rightarrow 0}T(t)Px=P^2x$. (One may simplify by taking $t=s$.)\\
We consider the spaces:
$$
E_1= R(PA^{-1})=\{x\in X,\  PA^{-1}x= A^{-1}x\}
$$
and
$$
E_2=Ker(PA^{-1})=\{x\in X; \ PA^{-1}x=0\}.
$$

Denoting again by $T_k$ the restriction of $T_k$ to $E_k$ $(k=1,2),$
we have:
\begin{lemma}\label{splitting-lemma}
    The operators $A_1$ and $A_2$ in Proposition \ref{proposition_splitting}
    satisfy the following:\\
    i)  $A_1$  generates the  holomorphic semigroup $ T_1$ on $E_1.$\\
    ii) $-A_2$  generates a holomorphic semigroup $T_2$ on $E_2.$
$\sigma(A_1)=\sigma_1$ and $\sigma(A_2)=\sigma_2.$
\end{lemma}
\begin{proof}
    i) For all $x\in E_1$ and $ {\rm   Re}(z)>\delta $ we have $$
    \int_0^\infty\exp(-zt)T_1(t)A_1^{-1}xdt=\frac{1}{2\pi i}\int_0^\infty\int_\Gamma\exp(-zt)\exp(\lambda t)
    \frac{R(\lambda,A)x+A^{-1}x}{-\lambda}d\lambda dt$$
$$=-\frac{1}{2\pi
i}\int_\Gamma\int_0^\infty\exp((-z+\lambda)t)\frac{R(\lambda,A)x}{\lambda}dtd\lambda-\frac{1}{2\pi
i}\int_\Gamma\int_0^\infty\exp((-z+\lambda)t)\frac{A^{-1}x}{\lambda}dt
d\lambda
$$
$$=-\frac{1}{2\pi
i}\int_\Gamma\int_0^\infty\exp((-z+\lambda)t)\frac{R(\lambda,A)x}{\lambda}dt
\lambda-\frac{1}{2\pi
i}\int_\Gamma\int_0^\infty\exp((-z+\lambda)t)\frac{A^{-1}x}{\lambda}dt
d\lambda$$
$$=\frac{1}{2\pi
i}\int_\Gamma\frac{R(\lambda,A)x}{\lambda (-z+\lambda)}
d\lambda+\frac{1}{2\pi
i}\int_\Gamma\frac{A^{-1}x}{\lambda(-z+\lambda)}
d\lambda=\int_\Gamma\frac{R(\lambda,A)x}{\lambda (-z+\lambda)}
d\lambda+\frac{A^{-1}x}{z}.$$ From which we deduce that $$(z-A_1)
    \int_0^\infty\exp(-zt)T_1(t)A_1^{-1}xdt=PA^{-1}x=A^{-1}x,
$$ which means that $R(z,A_1)A_1^{-1}x=
    \int_0^\infty\exp(-zt)T_1(t)A_1^{-1}xdt$ if   $z\in
    \rho(A_1).$
    It follows from the Phragmen-Lindel\"of Theorem that
    $$\sup_{z\in \rho(A_1), {\rm   Re}(z)>0}\Vert
    zR(z,A_1)\Vert<\infty.$$  Then $\{ {\rm   Re}(z)>0\}\subset
    \rho(A_1)$ and $\sup_{ {\rm   Re}(z)>0}\Vert
    zR(z,A_1)\Vert<\infty.$ The claim follows by the well known Phragm\'en-Lindel\"{o}f Theorem (see \cite{SW71} or \cite{Ti51}).\\

    ii) Analogously, considering $-A$ instead of $A$  and using a the change
    of variable $\lambda \rightarrow -\lambda $we obtain $-A_2$ is the generator of an holomorphic semigroup with angle $\pi /2$ in
    $E_2$.\\
    iii) From (i) we conclude $\sigma(A_1)\subset \sigma_1$.
   A similar purpose prove that
    $\sigma(A_2) \subset\sigma_2.$\\
    It suffices now to prove that $\sigma(A_1)\supseteq\sigma_1.$ If it is not the case, there exists
    $\lambda_1\in \C$ such that $\lambda_1\in \sigma_1\verb|\|\sigma(A_1).$ This implies that
    $\lambda_1 \in \rho(A_1)$. To obtain a contradiction, we prove that  $\lambda_1\in \rho(A).$
    Let us establish that it is injective and surjective:

    \emph{$\lambda_1-A$ is injective}: for $x\ne 0$ verifying $Ax=\lambda_1x,$ we have for all
     $\lambda\in \Gamma$ the identity $R(\lambda,A)x=\frac{x}{\lambda-\lambda_1}$ so

    \begin{eqnarray*}
        PA^{-1}x&=& \frac{1}{2i\pi}\int_{\Gamma}\frac{R(\lambda,A)x}{\lambda}d\lambda+A^{-1}x\\
        &=& \frac{1}{2i\pi}\int_{\Gamma}\frac{x}{\lambda(\lambda_1-\lambda)}d\lambda+A^{-1}x\\
        &=& -\frac{x}{\lambda_1}+A^{-1}x\\
        &=&0.
    \end{eqnarray*}

    But $PA^{-1}x=0$ implies that $x\in E_1$ so $Ax=A_1x=\lambda_1x.$ Since $\lambda_1\in \rho(A_1)$ then $x=0.$\\

    The operator \emph{$\lambda_1-A$ is surjective}:
    To prove that it is surjective, we will construct a preimage   $x$  under  $\lambda_1-A$    for an arbitrarily chosen $y\in E.$  It suffices to consider
    \begin{equation}\label{surjectivity}
    x=-A^{-1}y+\lambda_1 R(\lambda_1,A_1)(A^{-1}y-PA^{-1}y)+\lambda_1 R(\lambda_1,A_2)(PA^{-1}y).
    \end{equation}
    The identity (\ref{surjectivity})  is well defined because
    $A^{-1}y-PA^{-1}y\in E_1$ and $PA^{-1}y\in E_2.$ It is easy to verify that $(\lambda_1-A)x=y.$\\
    In general, the sum $E_1\oplus E_2$ in Proposition (\ref{proposition_splitting}) is not always closed as the  examples  below prove. The  example below (i.e Example \ref{H1example})  is based on an interesting result of \emph{harmonic analysis} due to D. J. Newman \cite{New61}.

\end{proof}

\begin{example}\label{H1example}
Let $E=L^1_{2\pi} $ and $A$ be given by $Af:=-if'$ with domain
$D(A):=\{f\in L^1_{2\pi};\ f'\in L^1_{2\pi}\}.$ Then $iA$
generates the $C_0-$group $T$  given by $T(it)f(x):=f(x+t).$ We
have: $E_1=\{ f\in L^1_{2\pi};\ \int_0^{2\pi} e^{-inx}f(x)dx=0 \hbox{
    for }
n\le 0\}=H^1$ and $E_2=\{ f\in L^1_{2\pi};\ \int_0^{2\pi} e^{-inx}f(x)dx=0
\hbox{ for } n> 0\}.$
  Moreover, by  a result due to D. J.  Newman  \cite{New61}, there is
no bounded  projection of $L^1_{2\pi}$ onto the Hardy space $H^1.$
It follows that $E_1 \oplus E_2\neq E.$ For more information about this important result on can see also \cite{Rud69} or more recently \cite{Ri00}.

\end{example}

%
%
%

\end{document}